\documentclass{amsart}

\newif\ifdraft\draftfalse
\newif\ifcite\citefalse
\newif\ifblow\blowfalse

\makeatletter
\@namedef{subjclassname@2020}{\textup{2020} Mathematics Subject Classification}
\makeatother

\usepackage{amsfonts,amssymb, amsmath}
\usepackage{url}
\usepackage{amscd}
\usepackage{mathrsfs} 
\usepackage[hyperindex,hidelinks]{hyperref}  

\ifdraft\ifcite\usepackage{showkeys}\else\usepackage[notcite,notref]{showkeys}\fi\fi

\newtheorem{proposition}[equation]{Proposition}
\newtheorem{theorem}[equation]{Theorem}

\newtheorem{lemma}[equation]{Lemma}

\theoremstyle{definition}

\theoremstyle{remark}

\numberwithin{equation}{section}

\newcommand\beq{\begin{equation}}
\newcommand\eeq{\end{equation}}

\newcommand\p{\partial}

\newcommand{\R}{\mathbb R}
\newcommand{\C}{\mathbb C}

\newcommand{\Id}{\mathrm{Id}}

\newcommand{\diag}{\operatorname{diag}}

\newcommand{\ip}[2]{\left\langle #1,#2\right\rangle}

\newcommand\al{\alpha}

\newcommand\lb{\label}

\newcommand\er{\eqref}

\newcommand\la{\langle}
\newcommand\ra{\rangle}

\newcommand\bc{{\mathbb C}}

\newcommand\br{{\mathbb R}}

\newcommand\fg{\mathfrak{g}}
\newcommand\fh{\mathfrak{h}}

\newcommand\fn{\mathfrak{n}}

\newcommand\fs{\mathfrak{s}}

\makeatletter
\@namedef{subjclassname@2020}{\textup{2020} Mathematics Subject Classification}
\makeatother
\usepackage[T1]{fontenc}
\usepackage{lmodern,amsmath,amssymb,mathtools,booktabs}
\usepackage{microtype}
\numberwithin{equation}{section}
\newcommand{\g}{\mathfrak g}
\newcommand{\h}{\mathfrak h}
\newcommand{\n}{\mathfrak n}
\newcommand{\s}{\mathfrak s}
\newcommand{\dd}{\,dA}
\newcommand{\pb}{\bar\partial}
\newcommand{\Ad}{\operatorname{Ad}}
\newcommand{\osc}{\operatorname{osc}}
\newcommand{\rh}{\rho^\vee}

\newcommand{\norm}[1]{\lVert#1\rVert}
\allowdisplaybreaks
\title[Analysis of blowup masses in constant-coefficient Toda systems]{Analysis of blowup masses in constant-coefficient Toda systems: W-invariants and Weyl orbits}
\author{Zhaohu Nie}
\address{Department of Mathematics and Statistics, Utah State University, Logan, Utah 84322-3900, USA}
\email{zhaohu.nie@usu.edu}
\begin{document}
\begin{abstract}
We classify the local blowup masses of constant-coefficient Toda systems for all simple Lie algebras with one prescribed singular source, bounded component masses, and bounded boundary oscillations. We encode the source strengths and the blowup masses by two elements of the Cartan subalgebra: the grading element $p_\mu$ and the mass element $q_\sigma$. We show that $p_\mu-q_\sigma=wp_\mu$ for some Weyl group element $w\in W$. The proof uses  the Drinfeld-Sokolov gauge and the resulting holomorphic $W$-invariants for Toda systems. Asymptotic analysis near the source and on a shrinking circle enclosing the total local mass gives two evaluations of the leading Laurent coefficients of the $W$-invariants, which agree by holomorphicity. This agreement enables us to conclude the above Weyl-orbit result. In particular, the quadratic constraint is the local Pohozaev identity. Only with the full set of $W$-invariants can we strengthen earlier orbit-plus-integer restrictions to an exact Weyl-orbit characterization in the constant-coefficient setting. Together with the author's explicit families in \cite{Nie-masses}, this determines all attainable local mass vectors.
\end{abstract}
\maketitle

\section{Introduction and main results}\label{sec:intro}

Blowup masses of Toda systems are closely related to the Weyl group of the underlying Lie algebra. In \cite{Nie-masses}, 
the author constructed explicit radial families realizing the local mass vector associated with each Weyl group element.
In this paper, we prove the converse for constant-coefficient systems: under the usual mass and boundary-oscillation bounds, these are all possible local masses. The proof uses the holomorphic $W$-invariants of the Toda equations.

Let $A=(a_{ij})_{i,j=1}^n$ be a finite simple Cartan matrix. We consider sequences of solutions $u^k=(u_1^k,\ldots,u_n^k)$ of the Toda system on the unit disk $B_1\subset\C$:
\begin{equation}\label{eq:toda}
 \Delta u_i^k+4\sum_{j=1}^n a_{ij}e^{u_j^k}
 =4\pi\gamma_i\delta_0,\qquad \gamma_i>-1.
\end{equation}
Here
\[
 \Delta=\p_x^2+\p_y^2=4\p\pb,\qquad
 \p=\frac{\p}{\p z}=\p_z=(\p_x-i\p_y)/2.
\]
The measure $\delta_0$ is the unit point mass at zero, and the source strengths $\gamma_i$ are fixed. We use the usual solution class: $u_i^k$ is smooth away from zero and the regular part
\[
 \hat u_i^k=u_i^k-2\gamma_i\log|z|
\]
extends continuously across zero. We assume continuous boundary values on $\p B_1$.

A sequence $u^k$ of solutions of \eqref{eq:toda} is said to blow up at the origin if there exist points $x_k\to0$ such that
\begin{equation}\label{blowup}
 \lim_{k\to\infty}\max_{1\le i\le n}\hat u_i^k(x_k)=+\infty.
\end{equation}
After taking a subsequence, we define the local masses by
\begin{equation}\label{eq:mass}
 \sigma_i=\frac1\pi\lim_{r\downarrow0}\lim_{k\to\infty}
          \int_{B_r}e^{u_i^k}\dd.
\end{equation}
(See also \eqref{eq:mass-bound} below.) 
The iterated limit is the atom at zero of the weak limit of $\pi^{-1}e^{u_i^k}dA$. 
We call 
\beq\lb{ma ve}
\sigma=(\sigma_1,\dots,\sigma_n)
\eeq
the local mass vector
of the sequence at zero.

The blowup analysis of exponential equations has a long history, beginning with the estimates of Brezis--Merle \cite{BM} and the quantization theorem of Li--Shafrir \cite{LS}. For Toda systems, Lin--Wei--Zhang \cite{LWZ} classified the local masses for singular $SU(3)$ systems, and Lin--Wei--Yang--Zhang \cite{LWYZ} developed a unified analysis for rank two. We use Lin--Yang--Zhong \cite{LYZ} as our main reference for the local blowup framework in types $A_n$, $B_n$, $C_n$, and $G_2$. Affine analogues for types $A$ and $C^t$ were obtained by Cui--Nie--Yang \cite{CNY}.

The finite-type results in \cite{LYZ} allow variable positive coefficients. In our normalization, their local equation has the form (sum index runs from 1 to $n$ in this paper, unless otherwise specified)
\begin{equation}\label{eq:variable}
 \Delta u_i^k+4\sum_j a_{ij}h_j^k(x)e^{u_j^k}
 =4\pi\gamma_i\delta_0,
\end{equation}
where $h_i^k(0)=1$, $C^{-1}\le h_i^k\le C$, and $\|h_i^k\|_{C^3(B_1)}\le C$.
The substitution $v_i^k=u_i^k+\log h_i^k$ gives
\begin{equation}\label{eq:variable-shift}
 \Delta v_i^k+4\sum_j a_{ij}e^{v_j^k}
 =4\pi\gamma_i\delta_0+\Delta\log h_i^k.
\end{equation}
The term $\Delta\log h_i^k$ in \er{eq:variable-shift} generally destroys the holomorphicity of the $W$-invariants to be introduced in \er{w inv}. Extending our argument to \er{eq:variable} 
requires further estimates.

Our local hypotheses on the solution sequence are the constant-coefficient specialization of \cite[(1.15)--(1.18)]{LYZ}:
\begin{align}
 &\sup_k\sup_K u_i^k<\infty
   &&\text{for every compact }K\Subset B_1\setminus\{0\},\label{eq:isolated}\\
 &\osc_{\p B_1}u_i^k\le C,
   &&1\le i\le n,\label{eq:boundary}\\
 &\int_{B_1}e^{u_i^k}\dd\le C,
   &&1\le i\le n.\label{eq:mass-bound}
\end{align}
Condition \er{eq:isolated} says that zero is the only possible limiting blowup point, and condition \er{eq:mass-bound} says that components have bounded masses. 
In \eqref{eq:boundary}, $\osc_{\p B_1}u_i^k=\sup_{\p B_1}u_i^k-\inf_{\p B_1}u_i^k$, so the condition bounds the variation of each component on the boundary, without bounding its average. In Green's representation \er{old rep}, this condition controls
the positive-order interior derivatives of the harmonic boundary
term $H_i^k$. The condition cannot simply be dropped. Chen's example \cite{Chen} gives a smooth Liouville sequence with local mass $2$ in our normalization. This mass satisfies the orbit-plus-integer restriction, but violates the rank-one Pohozaev identity $\sigma^2=\sigma$. We explain the unbounded boundary oscillation of this example in Section~\ref{sec:chen}.

The equation in Lin--Yang--Zhong \cite{LYZ} has nonlinear coefficient $1$ rather than $4$, and its masses are normalized by $1/(2\pi)$; the conversion is that their masses are twice ours in this paper: 
\begin{equation}\label{eq:LYZ-normalization}
 \sigma_i^{\rm LYZ}=2\sigma_i.
\end{equation}

Condition \eqref{eq:isolated} still allows several concentrating centers and scales to collapse to zero. The analysis in \cite{LYZ} uses the usual approach to select bubbling disks, rescale the solutions, and follow the mass changes as the disks are enlarged and grouped. 
In type $A_n$, their Theorem~1.5 gives an orbit-plus-even-integers restriction. 
Our approach is different. We choose a shrinking circle outside the entire concentrating region and read the total mass from the logarithmic asymptotics there (see Proposition~\ref{two rates}). Comparing the full set of $W$-invariants on this circle with their values at the prescribed source gives the Weyl-orbit restriction (see Proposition~\ref{thm:constraints}), hence strengthening their result by eliminating the integer ambiguities. In Section \ref{sec:rank-one}, we give an example with two concentrating centers but still the mass is determined by the Weyl orbit structure. 

Let us now describe our results more precisely. Let $\g$ be the complex simple Lie algebra with Cartan matrix $A$. Let $\fh_\bc$ be its complex Cartan subalgebra, and $\fh_\br$ the real form. We write the pairing $\fh_\bc^*\times\fh_\bc\to \bc$ either directly or via $\la\cdot,\cdot\ra$. 

Let $\{\al_1,\dots,\al_n\}$ be the simple roots and $\{h_1,\dots,h_n\}$ the simple coroots such that
\begin{equation}\label{eq:cartan}
 a_{ij}=\alpha_i(h_j),\qquad h_j=\alpha_j^\vee.
\end{equation}
Let $\omega_i$ be the fundamental weights, so $\omega_i(h_j)=\delta_{ij}$. Introduce the fundamental coweights
\[
E_j=\sum_i (A^{-1})_{ij}h_i\in\h_{\R},
\qquad\text{so that}\qquad \alpha_i(E_j)=\delta_{ij}.
\]
Put $\rh=\sum_iE_i$, so $\alpha_i(\rh)=1$. We encode the source strengths and masses by
\begin{equation}\label{eq:parameters}
 \begin{gathered}
 \gamma^i=\sum_j(A^{-1})_{ij}\gamma_j,\qquad
 p_\gamma=\sum_i\gamma_iE_i=\sum_i\gamma^i h_i,\\
 p_\mu=p_\gamma+\rh=\sum_i\mu_iE_i,\qquad
 q_\sigma=\sum_i\sigma_i h_i,
 \end{gathered}
\end{equation}
where $\mu_i=1+\gamma_i>0$, and $\sigma_i$ is the mass in \er{eq:mass}. 
We call $p_\mu$ the (general) grading element, and it  satisfies $\alpha_i(p_\mu)=\mu_i>0$.  
Denote the Weyl group by $W$.

We define the shifted Weyl action on $\h_{\R}$ by
\begin{equation}\label{eq:mass-action}
 w\star_\mu q:=p_\mu-w(p_\mu-q),
\end{equation}
where $w\in W$ and $q\in \h_\R$. 
This is an affine action of the finite Weyl group $W$, with fixed point $p_\mu$. When there are no singularities, $p_\mu=\rho^\vee$; under $q=-H$, this becomes the affine action $H\mapsto w(H+\rho^\vee)-\rho^\vee$, which is the dot action for the dual root system; see \cite{Humphreys-O}.

\begin{theorem}\label{thm:main}
Under \eqref{eq:toda}, \eqref{eq:isolated}, \eqref{eq:boundary}, and \eqref{eq:mass-bound}, there exists $w\in W$ such that
\begin{equation*}
 p_\mu-q_\sigma=wp_\mu,\quad i.e.,\ q_\sigma = w\star_\mu 0. 
\end{equation*}
Equivalently, the local mass vector \er{ma ve} belongs to the finite set
\begin{equation}\label{eq:Gamma}
 \Gamma_A(\mu)=
 \left\{\bigl(\omega_i(p_\mu-wp_\mu)\bigr)_{i=1}^n:w\in W\right\}.
\end{equation}
\end{theorem}

Theorem~\ref{thm:main} says that $q_\sigma$ belongs to the orbit
\begin{equation*}
 \mathcal Q_\mu:=W\star_\mu0
 =\{p_\mu-wp_\mu:w\in W\}.
\end{equation*}
The set $\Gamma_A(\mu)$ in \eqref{eq:Gamma} consists of the coordinates of these mass elements in the simple-coroot basis. In type $A_n$, the restriction in \cite[Theorem~1.5]{LYZ}, under their variable-coefficient hypotheses and the normalization \eqref{eq:LYZ-normalization}, reads $q_\sigma\in\mathcal Q_\mu+Q^\vee$, where $Q^\vee=\bigoplus_i\mathbb Z h_i$ is the coroot lattice.

Therefore, for the subcase of constant-coefficient Toda systems \er{eq:toda}, our paper improves \cite{LYZ} in two ways. First, we remove the coroot-lattice ambiguity. Second, our result holds for all finite simple Lie algebras, rather than their restriction to $A, B, C, G$ types. Furthermore, our method is totally different in utilizing outer-annulus asymptotics and $W$-invariants. 

The families of \cite{Nie-masses} realize every vector in \eqref{eq:Gamma}. Thus Theorem~\ref{thm:main} determines the full local mass set for constant-coefficient Toda systems under the conditions \eqref{eq:isolated}, \eqref{eq:boundary}, and \eqref{eq:mass-bound}. We verify the local hypotheses for these families in Section~\ref{sec:realization}.

Here are the key steps in the proof. With $U_i^k=\sum_j(A^{-1})_{ij}u_j^k$, \er{eq:toda} becomes  
\begin{equation}\label{eq:dual}
 \Delta U_i^k=-4e^{u_i^k}+4\pi\gamma^i\delta_0.
\end{equation}
The two logarithmic coefficients to be compared are $\gamma^i$ at the source and $\gamma^i-\sigma_i$ outside the whole atom. The following is the analytical core of this paper. 

\begin{proposition}[Source and outer asymptotics]\label{two rates}
\begin{enumerate}
\item\label{part:source} For each fixed solution of \eqref{eq:toda} and each integer $m\ge1$,
\begin{equation}\label{eq:source-jets}
 r^m\p^mU_i(ry)\longrightarrow
 (-1)^{m-1}(m-1)!\frac{\gamma^i}{y^m}
 \quad(r\to 0).
\end{equation}
\item\label{part:outer} For a sequence $u_i^k$ of solutions of \er{eq:toda} satisfying the hypotheses of Theorem~\ref{thm:main} with local masses $\sigma_i$ \er{eq:mass}, after passage to a subsequence there exist radii $r_k\to 0$ enclosing the total local mass in the sense of Lemma~\ref{lem:radii}, such that, for each integer $m\ge1$,
\begin{equation}\label{eq:outer-jets}
 r_k^m\p^mU_i^k(r_ky)\longrightarrow
 (-1)^{m-1}(m-1)!\frac{\gamma^i-\sigma_i}{y^m}
 \quad(k\to\infty).
\end{equation}
\end{enumerate}
Both convergences hold locally uniformly on $\C^*$, with all further derivatives on compact annuli. 
\end{proposition}

The outer-annulus asymptotic in Proposition~\ref{two rates}(\ref{part:outer}) is the technical core of the argument. Its proof uses Green's formula and the logarithmic-potential estimate underlying the Brezis--Merle method \cite{BM}. The entire-solution argument in \cite[Proposition~5.1]{KLNW} uses control at zero and infinity, and \cite[Theorem~1.1]{KLNW} gives the analogue of \er{eq:outer-jets} at infinity.  Here the outer-annulus estimates provide the corresponding local comparison.

We recall the construction of $W$-invariants following \cite[Section~4]{KLNW}; see also \cite{DS,Nie}. Let $\phi_1,\ldots,\phi_n$ be smooth functions on a domain $\Omega\subset\C$, and put
\(
 f=\sum_i e_{-\alpha_i},
\)
where $e_{-\al_i}$ is the root vector for the negative simple root. 
Let 
\beq\lb{grading}
\g=\bigoplus_m\g_m,\quad [\rh,X]=mX \text{ for }X\in\g_m,
\eeq
be the principal grading. Choose a homogeneous Kostant slice $\s\subset\n_+=\oplus_{m\geq 1}\fg_m$ and a basis such that
\beq\lb{slice}
 \g=[f,\g]\oplus\s,\qquad
 \s=\operatorname{span}\{s_1,\ldots,s_n\},\quad s_a\in\g_{d_a-1},
\eeq
where $d_1,\ldots,d_n$ are the degrees of the basic homogeneous invariant polynomials \cite{Kostant}, counted with multiplicity. There is a unique $M:\Omega\to N_+$, where $N_+$ is the Lie group corresponding to the nilpotent Lie algebra $\fn_+$, such that
\begin{equation*}
 M(z)(\p+f-\sum_i\phi_i h_i)M(z)^{-1}=\p+f+\sum_a S_a(\phi)s_a.
\end{equation*}
Each coefficient $S_a(\phi)$ is a differential polynomial in the $\phi_i$ and their $z$-derivatives, depending only on the chosen slice basis. Assigning weight $m+1$ to $\p^m\phi_i$, the polynomial $S_a$ is homogeneous of weight $d_a$. Furthermore, the coordinates of $M$ are successively defined, differential polynomials in the $\phi_i$, and 
\beq\lb{homog}
\text{the  coordinate of }M\text{ at a root of height }m 
\text{ has weight }m.
\eeq

For a solution of \er{eq:toda}, take $\phi_i=\p U_i$. Then
\beq\lb{w inv}
 W_a[U]:=S_a(\p U_1,\ldots,\p U_n)
\eeq
is holomorphic away from the sources by \cite[Lemma~4.1(iii)]{KLNW}. These are the $W$-invariants of the solution.

Finally, for $H=\sum_i b_i h_i\in\h_\C$ with $b_i\in \bc$, take $\phi_i=b_i/z$. Since
\[
 \p^m\phi_i=(-1)^m m!\frac{b_i}{z^{m+1}},
\]
homogeneity gives $S_a(\phi)=c_a(H)z^{-d_a}$, where $c_a$ is a polynomial on $\h_\C$. The corresponding gauge $M_H$ satisfies
\begin{equation}\label{def ca}
 M_H(z)\left(\p+f-\frac H z\right)M_H(z)^{-1}
 =\p+f+\sum_a\frac{c_a(H)}{z^{d_a}}s_a.
\end{equation}
Let $I_a$ be homogeneous basic Weyl-invariant polynomials of degrees $d_a$ on $\h_\C$ \cite{Kostant}.
We have the following algebraic equivalence. 

\begin{proposition}\label{thm:constraints}
Let $H, H'\in \fh_\bc$. The following three conditions are equivalent. 
\begin{enumerate}
\item\label{part:residues} the residue identities
\begin{equation}\label{eq:c-constraints}
 c_a(H)=c_a(H'),\qquad a=1,\ldots,n;
\end{equation}
\item\label{part:invariants} the invariant identities
\begin{equation}\label{eq:invariant-constraints}
 I_a(H+\rh)=I_a(H'+\rh),\qquad a=1,\ldots,n;
\end{equation}
\item\label{part:orbit} the Weyl-orbit identity
\beq\lb{gen Weyl}
H'+\rh = w(H+\rh),\quad\exists w\in W.
\eeq
\end{enumerate}
\end{proposition}


The holomorphicity of the $W$-invariant \er{w inv} and Proposition \ref{two rates}, together with Proposition \ref{thm:constraints}, give the proof of our main Theorem \ref{thm:main}. 

The paper is organized as follows. Section~\ref{sec:annulus} proves the inner and outer asymptotics in Proposition \ref{two rates}. Section~\ref{sec:currents} proves conservation of the leading Laurent coefficients, identifies the resulting Weyl orbits, and verifies realization of the mass vectors. The quadratic restriction in Section~\ref{sec:pohozaev} recovers the local Pohozaev identity. Section~\ref{sec:examples} gives the examples discussed above and an $A_2$ constraint from the cubic invariant.

\medskip
\noindent{\bf Acknowledgment of AI assistance.}
As a Plus user, the author used ChatGPT for mathematical discussion, assistance with
checking computations, and editorial suggestions during the preparation
of this paper. The author is responsible for its mathematical content
and final presentation.

\section{The two asymptotics}\label{sec:annulus}

In this section, we prove Proposition~\ref{two rates}. The inner asymptotics follow by rescaling a fixed solution near its prescribed source. For the outer asymptotics, we choose an annulus enclosing the whole atom in Lemma \ref{lem:radii} and show, in Lemma \ref{lem:small}, that its nonlinear density tends to zero after rescaling. Green's formula then gives the coefficient $\gamma^i-\sigma_i$. The mass is already encoded in the integral in \er{eq:green-gradient}; the estimates below justify replacing the concentrating measure by its atom when evaluating the rescaled integral.

In this section, $C$ denotes different constants, which may depend on the Cartan matrix, the source strengths, the bounds in \eqref{eq:boundary}--\eqref{eq:mass-bound}, and the compact sets under consideration.

\subsection{Inner asymptotics at the source}

\begin{proof}[Proof of Proposition~\ref{two rates}(\ref{part:source})]
We use the near-source argument of \cite[(5.10)--(5.11)]{LWY}, as adapted in \cite[proof of Proposition~5.1]{KLNW}, and give the rescaling details here. Fix one solution and set $V_i=U_i-2\gamma^i\log|z|$. The regular parts $V_i$ are continuous at zero. Choose compactly nested annuli
\[
 K\Subset\Omega'\Subset\Omega\Subset\C^*,
\]
and put $v_{i,r}(y)=V_i(ry)-V_i(0)$. Continuity gives $\|v_{i,r}\|_{L^\infty(\Omega)}\to0$ as $r\to 0$. The equation for the regular part becomes
\begin{equation}\label{eq:source-rescaled}
 \Delta_yv_{i,r}(y)=-4r^{2+2\gamma_i}|y|^{2\gamma_i}
       \exp\left(\sum_j a_{ij}\bigl(V_j(0)+v_{j,r}(y)\bigr)\right)
       =:F_{i,r}(y).
\end{equation}
Since $\Omega$ stays away from zero, the factors following $r^{2+2\gamma_i}$ are uniformly bounded. Thus $\|F_{i,r}\|_{L^\infty(\Omega)}\to0$, because $\gamma_i>-1$. For any fixed $p>2$, the interior estimate gives
\[
 \|v_{i,r}\|_{W^{2,p}(\Omega')}
 \le C\left(\|v_{i,r}\|_{L^p(\Omega)}
                  +\|F_{i,r}\|_{L^p(\Omega)}\right)\longrightarrow0.
\]
Sobolev embedding therefore gives $v_{i, r}\to 0$ in $C^{1,\alpha}(\Omega')$ for $0<\alpha<1-2/p$. The formula for $F_{i,r}$ in \er{eq:source-rescaled} then gives $F_{i,r}\to0$ in $C^{1,\alpha}(\Omega')$. Repeated interior Schauder estimates on smaller annuli yield $v_{i,r}\to0$ in every $C^m(K)$.

Consequently, for every $m\ge1$,
\[
 r^m\p_z^mV_i(ry)=\p_y^m v_{i,r}(y)\longrightarrow0.
\]
On the other hand,
\[
 \p_y^m\bigl(2\gamma^i\log|ry|\bigr)
 =(-1)^{m-1}(m-1)!\frac{\gamma^i}{y^m}.
\]
Adding the two terms proves \eqref{eq:source-jets}, with all further derivatives on compact annuli. The constants may depend on the fixed solution.
\end{proof}

\subsection{Green representation and the choice of outer radii}
\begin{lemma}\label{lem:green}
On $B_{1/2}\setminus\{0\}$ one has
\begin{equation}\label{eq:green-gradient}
 \p U_i^k(z)=\frac{\gamma^i}{z}
 -\frac1\pi\int_{B_1}\frac{e^{u_i^k(w)}}{z-w}\dd_w+g_i^k(z),
\end{equation}
where $g_i^k$ is holomorphic and uniformly bounded with all its derivatives on each smaller disk.
\end{lemma}
\begin{proof}
From \er{eq:dual} and using the Dirichlet Green function of $-\Delta$ on $B_1$,
\[
 G(z,w)=\frac1{2\pi}\log\frac{|1-z\bar w|}{|z-w|},
\]
write
\beq\lb{old rep}
 U_i^k(z)=2\gamma^i\log|z|+4\int_{B_1}G(z,w)e^{u_i^k(w)}\dd_w+H_i^k(z),
\eeq
where $H_i^k$ is the harmonic extension of $U_i^k|_{\p B_1}$. By \eqref{eq:boundary}, $H_i^k$ is uniformly bounded after subtraction of a constant. Its positive-order interior derivatives are therefore bounded.

Differentiating the Green kernel gives
\[
 \p_zG(z,w)=-\frac1{4\pi}
 \left(\frac1{z-w}+\frac{\bar w}{1-z\bar w}\right).
\]
Thus \eqref{eq:green-gradient} holds with
\[
 g_i^k(z)=\p H_i^k(z)-\frac1\pi\int_{B_1}
 \frac{\bar w}{1-z\bar w}e^{u_i^k(w)}\dd_w.
\]
The kernel in the last integral is holomorphic in $z$ with uniformly bounded derivatives for $|z|\le1/2$ and $|w|\le1$. The mass bound \eqref{eq:mass-bound} shows uniform bounds for all derivatives of $g_i^k$.
\end{proof}

\begin{lemma}[Separated radii]\label{lem:radii}
After passage to a subsequence there are $a_k,r_k,b_k\downarrow0$ such that
\begin{equation}\label{eq:separated}
 \frac{a_k}{r_k}\longrightarrow0,\qquad
 \frac{r_k}{b_k}\longrightarrow0,
\end{equation}
and, simultaneously for all components,
\begin{equation}\label{eq:annulus-mass}
 \frac1\pi\int_{B_{a_k}}e^{u_i^k}\dd\longrightarrow\sigma_i,
 \qquad \int_{B_{b_k}\setminus B_{a_k}}e^{u_i^k}\dd\longrightarrow0.
\end{equation}
\end{lemma}
\begin{proof}
By the mass bound \er{eq:mass-bound}, pass to a subsequence such that the finite measures $\pi^{-1}e^{u_i^k}dA$ on $\overline B_{1/2}$ converge weakly to $\nu_i$, simultaneously for all components:
\begin{equation*}
 \frac1\pi e^{u_i^k}dA\rightharpoonup\nu_i.
\end{equation*}
See \cite[Theorem~5.1]{Billingsley}. Only countably many circles can have positive mass for any of the $\nu_i$. For every remaining radius $0<r<1/2$, weak convergence gives
\begin{equation}\label{eq:disk-mass-limit}
 \frac1\pi\int_{B_r}e^{u_i^k}\dd\longrightarrow\nu_i(B_r)
\end{equation}
by \cite[Theorem~2.1]{Billingsley}. As these radii decrease to zero, $\nu_i(B_r)\to\nu_i(\{0\})=\sigma_i$ by \er{eq:mass}.

Choose such continuity radii $b_m\downarrow0$ and $a_m\downarrow0$ with $a_m<b_m/m^2$. Choose $k_m>k_{m-1}$ so that, for every component, the normalized masses in $B_{a_m}$ and $B_{b_m}$ differ from their limiting values in \er{eq:disk-mass-limit} by at most $1/m$. Both masses then tend to $\sigma_i$, so the mass of $B_{b_m}\setminus B_{a_m}$ tends to zero. Setting $r_m=\sqrt{a_mb_m}$ gives
\[
 \frac{a_m}{r_m}=\frac{r_m}{b_m}
 =\sqrt{\frac{a_m}{b_m}}<\frac1m.
\]
Relabeling this subsequence proves \eqref{eq:separated} and \eqref{eq:annulus-mass}.
\end{proof}

\subsection{Vanishing of the rescaled density}
\begin{lemma}\label{lem:small}
For the radii of Lemma~\ref{lem:radii}, set
\beq\lb{vrho}
 v_i^k(y)=u_i^k(r_ky)+2\log r_k,\qquad
 \rho_i^k(y)=e^{v_i^k(y)}=r_k^2e^{u_i^k(r_ky)}.
\eeq
For every compact annulus $K\Subset\C^*$ and every nonnegative integer $m$,
\begin{equation}\label{eq:density-small}
 \norm{\rho_i^k}_{C^m(K)}\longrightarrow0.
\end{equation}
Moreover the derivatives of $v_i^k$ of every positive order are uniformly bounded on $K$.
\end{lemma}
\begin{proof}
Choose bounded concentric annuli
\[
 K\Subset\Omega''\Subset\Omega'\Subset\Omega\Subset\C^*.
\]
By \eqref{eq:separated}, $r_k\Omega\subset B_{b_k}\setminus B_{a_k}$ for large $k$. Hence by \er{eq:annulus-mass}, 
\beq\lb{small rho}
 \epsilon_k=\sum_j\int_\Omega\rho_j^k\dd =\sum_j\int_\Omega e^{v_j^k}\dd = \sum_j \int_{r_k\Omega} e^{u_j^k}\dd\longrightarrow0.
\eeq

On $\Omega$, the rescaled equation is
\[
 \Delta v_i^k=-4\sum_j a_{ij}\rho_j^k.
\]

Multiplying \er{old rep} by $a_{ij}$ and summing over $j$, we write
\[
 u_i^k(x)=2\gamma_i\log|x|+b_i^k(x)
 -\frac2\pi\sum_j a_{ij}\int_{B_1}\log|x-z|e^{u_j^k(z)}\dd_z,
\]
where $b_i^k$ is harmonic on $B_1$ and includes the harmonic boundary extension and the regular part of the Green kernel. By Lemma~\ref{lem:green},
$\p b_i^k=\sum_j a_{ij}g_j^k$, so its positive-order derivatives are uniformly bounded on smaller disks.
Set $x=r_ky$ and split the integral over $r_k\Omega$ and $B_1\setminus r_k\Omega$. In the first part, the substitution $z=r_k\eta$ gives
\[
 \int_{r_k\Omega}\log|r_ky-z|e^{u_j^k(z)}\dd_z
 =\log r_k\int_\Omega\rho_j^k\dd
  +\int_\Omega\log|y-\eta|\rho_j^k(\eta)\dd_\eta.
\]
Consequently, by the definition \er{vrho}, 
\begin{equation}\label{eq:decomposition}
 v_i^k=c_i^k+H_i^k+R_i^k,\qquad
 R_i^k(y)=-\frac2\pi\sum_j a_{ij}
                 \int_\Omega\log|y-\eta|\rho_j^k(\eta)\dd_\eta.
\end{equation}
To specify the harmonic term and the constant, fix $y_0\in\Omega'$ and put
\begin{equation}\label{eq:harmonic-remainder}
 \begin{aligned}
 \widetilde H_i^k(y)
 &=2\gamma_i\log|y|+b_i^k(r_ky)\\
 &\quad-\frac2\pi\sum_j a_{ij}
 \int_{B_1\setminus r_k\Omega}
       \log|r_ky-z|e^{u_j^k(z)}\dd_z,\\
 H_i^k(y)&=\widetilde H_i^k(y)-\widetilde H_i^k(y_0),\\
 c_i^k&=\left(2+2\gamma_i-\frac2\pi\sum_j a_{ij}
                  \int_\Omega\rho_j^k\dd\right)\log r_k
                  +\widetilde H_i^k(y_0).
 \end{aligned}
\end{equation}
Every term in $\widetilde H_i^k$ in \er{eq:harmonic-remainder} is harmonic on $\Omega$: the annulus avoids zero, and $z/r_k\notin\Omega$ in the exterior integral. For $y\in\Omega'$, the distance from $y$ to $z/r_k$ is bounded below. The mass bound \er{eq:mass-bound} therefore bounds the gradient in $y$ of the exterior potential. Together with the bound for $\p b_i^k$ and the normalization $H_i^k(y_0)=0$, this gives
\beq\lb{hi bound}
 \|H_i^k\|_{L^\infty(\Omega')}\le C.
\eeq
Harmonic interior estimates applied to \er{hi bound} then bound all derivatives of $H_i^k$ on smaller annuli.

For every finite $p\ge1$, Minkowski's inequality and \er{small rho} give
\begin{equation}\label{eq:R-Lp}
 \begin{aligned}
 \norm{R_i^k}_{L^p(\Omega')}
 &\le\frac2\pi\sum_j |a_{ij}|\int_\Omega
   \|\log|\cdot-\eta|\|_{L^p(\Omega')}
   \rho_j^k(\eta)\dd_\eta\\
 &\le C_p\epsilon_k,
 \end{aligned}
\end{equation}
since $\sup_{\eta\in\Omega}\|\log|\cdot-\eta|\|_{L^p(\Omega')}<\infty$.

We first control the constant $c_i^k$ by averaging, and then control $e^{R_i^k}$ by the small mass of its logarithmic potential. Using \er{hi bound} and the $p=1$ case of \er{eq:R-Lp}, averaging \eqref{eq:decomposition} and applying Jensen's inequality gives
\begin{equation}\label{eq:c-small}
 c_i^k\le \frac1{|\Omega'|}\int_{\Omega'}{v_i^k}\dd + C \le \log\left(\frac1{|\Omega'|}\int_{\Omega'}e^{v_i^k}\dd\right)+C
 \le\log\epsilon_k+C.
\end{equation}
This controls the additive constants.

To control $e^{R_i^k}$, fix $p>1$, and consider the off-diagonal terms in \eqref{eq:decomposition}.
Because $a_{ij}\le0$ for $j\ne i$ and $\log|y-\eta|$ is bounded above, 
$$
-\frac{2}{\pi}\sum_{j\neq i}a_{ij} \int_\Omega \log|y-\eta| \rho_j^k(\eta)\dd_\eta \le C\sum_{j\ne i}\int_\Omega\rho_j^k(\eta)\dd_\eta
\le C\epsilon_k. 
$$
For the diagonal term, set $m_i^k=\int_\Omega\rho_i^k$ and choose $C_0>\operatorname{diam}\Omega$. This makes the logarithmic kernel positive; its insertion changes the diagonal potential by a constant of size $O(m_i^k)$. Jensen's inequality, with the probability measure $\rho_i^k\dd/m_i^k$, gives
\begin{equation}\label{eq:BM}
 \exp\left(\frac{4p}{\pi}\int_\Omega
       \log\frac{C_0}{|y-\eta|}\rho_i^k(\eta)\dd_\eta\right)
 \le\frac1{m_i^k}\int_\Omega
       \left(\frac{C_0}{|y-\eta|}\right)^{4pm_i^k/\pi}
       \rho_i^k(\eta)\dd_\eta.
\end{equation}
Put $\theta_k=4pm_i^k/\pi$. For each $\eta\in\Omega$,
\beq\lb{om'}
 \int_{\Omega'}\left(\frac{C_0}{|y-\eta|}\right)^{\theta_k}\dd_y
 \le 2\pi C_0^{\theta_k}\int_0^{C_0}r^{1-\theta_k}\,dr
 =\frac{2\pi C_0^2}{2-\theta_k}.
\eeq
Thus $\theta_k<2$ suffices for integrability. Since $\theta_k\to0$, we may take $\theta_k\le1$ for large $k$, giving a uniform bound of \er{om'}. Integrating \eqref{eq:BM} over $\Omega'$ with respect to $y$ and using $\int_\Omega\rho_i^k/m_i^k=1$, together with the off-diagonal bound, yields 
\begin{equation}\label{eq:exp-R-bound}
\|e^{R_i^k}\|_{L^p(\Omega')}\le C_p,
\end{equation}
 as in \cite{BM}. Combining \er{eq:exp-R-bound}, \er{eq:c-small}, and \er{hi bound} in the decomposition \er{eq:decomposition} gives 
\begin{equation}\label{eq:Lp-small}
\norm{\rho_i^k}_{L^p(\Omega')}= \norm{e^{v_i^k}}_{L^p(\Omega')}\le C_p\epsilon_k.
\end{equation}

Since $\Delta R_i^k=-4\sum_j a_{ij}\rho_j^k$ on $\Omega'$, the interior elliptic estimate, together with \er{eq:R-Lp} and \er{eq:Lp-small}, gives, for fixed $p>2$,
\begin{equation}\label{eq:R-W2p}
 \|R_i^k\|_{W^{2,p}(\Omega'')}
 \le C_p\left(\|R_i^k\|_{L^p(\Omega')}
       +\sum_j\|\rho_j^k\|_{L^p(\Omega')}\right)
 \le C_p\epsilon_k.
\end{equation}
Sobolev embedding gives $R_i^k\to0$ in $C^{1,\alpha}(\Omega'')$ for $0<\alpha<1-2/p$. Put $w_i^k=v_i^k-c_i^k=H_i^k+R_i^k$. The harmonic bound \er{hi bound} and the potential estimate \er{eq:R-W2p} give a uniform $C^{1,\alpha}$ bound for $w_i^k$. Since $c_i^k$ is spatially constant,
\[
 \Delta w_i^k=-4\sum_j a_{ij}e^{c_j^k}e^{w_j^k}.
\]
By \er{eq:c-small}, the right side is uniformly bounded in $C^{1,\alpha}(\Omega'')$. Repeated interior Schauder estimates on smaller annuli give uniform $C^m(K)$ bounds for $w_i^k$. Finally,
\[
 \rho_i^k=e^{c_i^k}e^{w_i^k},\qquad
 \|\rho_i^k\|_{C^m(K)}\le C_m e^{c_i^k}
 \le C_m\epsilon_k\longrightarrow0.
\]
\end{proof}

\subsection{Outer asymptotics}
\begin{proof}[Proof of Proposition~\ref{two rates}(\ref{part:outer})]
From \eqref{eq:green-gradient}, we obtain
\begin{equation}\label{eq:rescaled-green}
r_k\p U_i^k(r_ky)
=\frac{\gamma^i}{y}
-\frac1\pi\int_{B_1}
  \frac{e^{u_i^k(w)}}{y-w/r_k}\,dA_w
+r_k g_i^k(r_ky).
\end{equation}
Split the integral in \er{eq:rescaled-green} into $B_{a_k}$, $B_{b_k}\setminus B_{a_k}$, and $B_1\setminus B_{b_k}$. On a compact annulus in $y$, the first piece is
\[
 \frac1\pi\int_{B_{a_k}}\frac{e^{u_i^k(w)}}{y-w/r_k}\dd_w
 =\frac{\sigma_i+o(1)}{y}+O(a_k/r_k).
\]
Indeed, uniformly on a fixed compact annulus in $y$,
\[
 \frac1{y-w/r_k}=\frac1y+O(a_k/r_k)\qquad (|w|\le a_k).
\]
By \er{eq:annulus-mass}, the normalized enclosed mass tends to $\sigma_i$, so integrating
the kernel expansion gives the stated estimate.
 For the exterior piece, $|w|\ge b_k$ and $r_k/b_k\to0$ give
\[
 \left|\frac{r_k}{r_ky-w}\right|\le\frac{2r_k}{|w|}\le\frac{2r_k}{b_k}.
\]
The exterior integral is therefore $O(r_k/b_k)$ by \er{eq:mass-bound}. For the middle piece, fix $\delta>0$ so that $B_\delta(y)$ stays in a compact annulus. On this ball, the density bound \er{eq:density-small} gives
\[
 \int_{B_\delta(y)}\frac{\rho_i^k(\eta)}{|y-\eta|}\,dA_\eta
 \le 2\pi\delta\sup_{B_\delta(y)}\rho_i^k=o(1).
\]
On its complement the kernel is bounded by $\delta^{-1}$ and the mass tends to zero by \eqref{eq:annulus-mass}. Also $r_kg_i^k(r_ky)\to0$ by Lemma~\ref{lem:green}. Substituting these bounds into \er{eq:rescaled-green} gives
\[
 \p_yU_i^k(r_ky)\longrightarrow\frac{\gamma^i-\sigma_i}{y},
\]
uniformly on compact annuli, proving \eqref{eq:outer-jets} for $m=1$.

Set $F_i^k(y)=U_i^k(r_ky)-2(\gamma^i-\sigma_i)\log|y|$ and normalize by
\[
 \widetilde F_i^k(y)=F_i^k(y)-F_i^k(y_0)
\]
at a fixed point $y_0$. Since $F_i^k$ is real-valued, the convergence 
$$\p_yF_i^k = r_k\p U_i^k(r_k y) - \frac{\gamma^i-\sigma_i}{y}\to0$$ controls its full gradient. Integrating along paths in a fixed annulus gives $\widetilde F_i^k\to0$ locally uniformly. Subtracting the constant preserves the equation:
\[
 \Delta_y\widetilde F_i^k=-4\rho_i^k.
\]
On nested compact annuli $K\Subset K'$, interior Schauder estimates give
\[
 \|\widetilde F_i^k\|_{C^{m+2,\alpha}(K)}
 \le C\left(\|\widetilde F_i^k\|_{C^0(K')}
             +\|\rho_i^k\|_{C^{m,\alpha}(K')}\right)\longrightarrow0.
\]
Here the H\"older convergence of $\rho_i^k$ follows from \er{eq:density-small}. This proves the remaining derivatives in \eqref{eq:outer-jets}.
\end{proof}

\section{W-invariants and invariant polynomials}\label{sec:currents}\label{sec:algebra}

\subsection{The shifted invariant polynomials}
\begin{proof}[Proof of Proposition~\ref{thm:constraints}]
Let $G$ be the adjoint group of $\g$. Chevalley restriction and Kostant's restriction theorem give two isomorphisms
\begin{equation}\label{eq:invariant-rings}
 \bc[\fh_\bc]^W\xleftarrow{\ \sim\ }\bc[\fg]^G
 \xrightarrow{\ \sim\ }\bc[f+\s],
\end{equation}
where the slice $\fs$ is defined in \er{slice}; 
see \cite{Kostant,Kostant-Whittaker}. We use the same symbols $I_a$ for the adjoint-invariant polynomials and their restrictions to $\fh_\bc$, so
$\bc[\fh_\bc]^W=\bc[I_1,\ldots,I_n]$.
The condition \er{eq:invariant-constraints} is independent of the choice of basic generators. By \er{eq:invariant-rings}, we may choose homogeneous generators normalized by
\begin{equation}\label{eq:slice-normalization}
 I_a\left(f+\sum_b t_bs_b\right)=t_a.
\end{equation}

The shift by $\rh$ below is the same gauge contribution that occurs in the residues of regular-singular opers; see \cite[Sections~9.1--9.2, especially Proposition~9.1.1]{Frenkel-loop}. We give the calculation in our conventions.

On a chosen branch of $\log z$, define
\[
 z^{\rh}:=\exp((\log z)\rh).
\]
Its adjoint action on $\g_m$ \er{grading} is multiplication by $z^m$, hence is single-valued.
By \eqref{homog}, the coordinate of $M_H$ at a root of height $m$ is a constant multiple of $z^{-m}$. Hence
\[
 N_H=z^{\rh}M_H(z)z^{-\rh}
\]
is constant. Using
\[
 z^{\rh}\p z^{-\rh}=\p-\frac{\rh}{z},\qquad
 \Ad(z^{\rh})f=z^{-1}f,\qquad
 \Ad(z^{\rh})s_a=z^{d_a-1}s_a
\]
in \eqref{def ca}, and comparing the coefficients of $z^{-1}$, gives
\begin{equation}\label{eq:residue-conjugacy}
 \Ad(N_H)(f-H-\rh)=f-\rh+\sum_a c_a(H)s_a.
\end{equation}
For every $I\in\bc[\fg]^G$ and $X\in\fh_\bc$,
\[
 I(f-X)=I(-X),
\]
since $\Ad(t^{-\rh})(f-X)=tf-X\to-X$ as $t\to0$.
Put $c(H)=(c_1(H),\ldots,c_n(H))$ and
\[
 F_a(t)=I_a\left(f-\rh+\sum_b t_bs_b\right).
\]
Assign weight $d_b$ to $t_b$. The polynomial
$I_a(f-u\rh+\sum_b t_bs_b)$ is weighted homogeneous of degree $d_a$ when $u$ has weight one, as follows by applying $\lambda\Ad(\lambda^{\rh})$ to its argument. At $u=0$ it equals $t_a$ by \er{eq:slice-normalization}. Consequently,
\begin{equation*}
 F_a(t)=t_a+P_a(t_b:d_b<d_a),
\end{equation*}
where $P_a$ has weighted degree strictly less than $d_a$. Thus 
\beq\lb{F iso}
F:\bc^n\to \bc^n\quad\text{is a triangular polynomial automorphism}
\eeq

Now \er{eq:residue-conjugacy} and homogeneity give
\begin{equation*}
 F_a(c(H))=I_a(f-H-\rh)=(-1)^{d_a}I_a(H+\rh).
\end{equation*}
Therefore, the automorphism \er{F iso} gives the equality 
\begin{equation*}
 \bc[c_1(H),\ldots,c_n(H)]
 =\bc[I_1(H+\rh),\ldots,I_n(H+\rh)],
\end{equation*}
proving the equivalence of \er{eq:c-constraints} and \er{eq:invariant-constraints}. 

Since invariant polynomials separate the orbits of the finite Weyl group, \er{eq:invariant-constraints} is equivalent to \er{gen Weyl}.
\end{proof}

\subsection{Conservation of the leading coefficients}
The $W$-invariants in \eqref{w inv} are holomorphic away from the sources by \cite{Nie} and \cite[Lemma~4.1(iii)]{KLNW}. The following local residue statement is the part of \cite[Proposition~5.1 and Theorem~5.3]{KLNW} needed here.

\begin{proposition}\label{prop:conservation}\label{lem:source}
For a fixed solution $u$ of \eqref{eq:toda}, put $U=A^{-1}u$. Then $W_a[U]$ is meromorphic at zero with pole order at most $d_a$, and
\begin{equation}\label{eq:source-residue}
 \lim_{z\to0}z^{d_a}W_a[U](z)=c_a(p_\gamma).
\end{equation}
Consequently, for every $0<r<1$, 
\begin{equation}\label{eq:contour}
 \frac1{2\pi i}\int_{|z|=r}z^{d_a-1}W_a[U](z)\,dz=c_a(p_\gamma).
\end{equation}
\end{proposition}
\begin{proof}
By \eqref{eq:source-jets}, for every $m\ge1$,
\[
 r^m\p^mU_i(ry)\longrightarrow
 (-1)^{m-1}(m-1)!\frac{\gamma^i}{y^m}
 \qquad(r\to0),
\]
uniformly on $|y|=1$. Each monomial in $W_a$ is a product of derivatives $\p^mU_i$ whose orders sum to $d_a$. Multiplying by $r^{d_a}$ therefore distributes one factor $r^m$ to each derivative. The limiting derivatives are those of the model with $\p U_i=\gamma^i/y$, so the definitions \er{w inv} and \er{def ca} give
\begin{equation}\label{eq:source-current-limit}
 r^{d_a}W_a[U](ry)\longrightarrow
 c_a(p_\gamma)y^{-d_a}
\end{equation}
uniformly on $|y|=1$.

Put $F_a(z)=z^{d_a}W_a[U](z)$. Multiplying \er{eq:source-current-limit} by $y^{d_a}$ gives $F_a(ry)\to c_a(p_\gamma)$ uniformly on $|y|=1$. Since every $z\ne0$ can be written as $z=ry$ with $r=|z|$ and $|y|=1$, this proves \eqref{eq:source-residue}. The function $F_a$ is holomorphic on $B_1\setminus\{0\}$ by the holomorphicity of the $W$-invariants and bounded near zero, hence extends holomorphically to $B_1$ with $F_a(0)=c_a(p_\gamma)$. Thus $W_a[U]=z^{-d_a}F_a$ is meromorphic with pole order at most $d_a$. For every $0<r<1$, Cauchy's formula now gives
\[
 \frac1{2\pi i}\int_{|z|=r}z^{d_a-1}W_a[U](z)\,dz
 =\frac1{2\pi i}\int_{|z|=r}\frac{F_a(z)}z\,dz
 =F_a(0)=c_a(p_\gamma),
\]
which is \eqref{eq:contour}.
\end{proof}

\begin{proof}[Proof of Theorem~\ref{thm:main}]
On the radii of Proposition~\ref{two rates}(\ref{part:outer}), weighted homogeneity and \eqref{eq:outer-jets} give, uniformly on $|y|=1$,
\beq\lb{rk conv}
 r_k^{d_a}W_a[U^k](r_ky)\longrightarrow c_a(p_\gamma-q_\sigma)y^{-d_a}.
\eeq
Apply \eqref{eq:contour} to each $U^k$ and set $z=r_ky$, with $r=r_k$. We obtain
\[
c_a(p_\gamma)
=\frac{1}{2\pi i}\int_{|y|=1}
y^{d_a-1}r_k^{d_a}W_a[U^k](r_ky)\,dy.
\]
By the uniform convergence \er{rk conv} on $|y|=1$, the right-hand side tends to
\[
\frac{1}{2\pi i}\int_{|y|=1}
y^{d_a-1}c_a(p_\gamma-q_\sigma)y^{-d_a}\,dy
=c_a(p_\gamma-q_\sigma).
\]
Thus \eqref{eq:c-constraints} holds with $H=p_\gamma$ and $H'=p_\gamma-q_\sigma$ as 
\beq\lb{cp}
c_a(p_\gamma)=c_a(p_\gamma-q_\sigma),\quad a=1,\dots, n.
\eeq

We comment that, in comparing the two asymptotics, only the leading Laurent coefficient is conserved; the lower coefficients may depend on $k$.

Proposition~\ref{thm:constraints}, with $p_\mu=p_\gamma+\rh$ from \er{eq:parameters}, gives $p_\mu-q_\sigma=wp_\mu$. Applying the fundamental weights gives \eqref{eq:Gamma}.
\end{proof}

\subsection{Realization of the mass vectors}\label{sec:realization}
For $w\in W$, the family in \cite[Theorem~1.7]{Nie-masses} with $\tau=w^{-1}$ has local masses
\[
 \sigma_i=\langle\omega_i-\tau\omega_i,p_\mu\rangle
          =\omega_i(p_\mu-wp_\mu).
\]
Radiality gives \eqref{eq:boundary}, and the fixed entire masses \cite[Theorem~1.1]{KLNW} give \eqref{eq:mass-bound}. With $H=\tau H_0$ and $H_0$ in the positive chamber, the fixed-annulus asymptotics of \cite[Section~3]{Nie-masses} give $u_i^{\lambda,H}=-2\lambda\alpha_i(H_0)+O_K(1)$ on $K\Subset\C^*$, proving \eqref{eq:isolated}. Thus every vector in \eqref{eq:Gamma} is attained, including zero. A nonzero atom forces blowup in the sense of \eqref{blowup}, since a uniform bound $e^{u_i^k}\le C|z|^{2\gamma_i}$ near zero would preclude concentration, by $\gamma_i>-1$. 

\subsection{The mass reflections}
The simple reflection is $s_iH=H-\alpha_i(H)h_i$, $H\in \fh_\bc$. The action $s_i\star_\mu q_\sigma$ in \eqref{eq:mass-action} is 
\begin{align*}
s_i\star_\mu q_\sigma &= p_\mu - s_i(p_\mu-q_\sigma) = p_\mu-(p_\mu-q_\sigma - \la \al_i, p_\mu-q_\sigma\ra h_i)\\
&=q_\sigma + \Big(\mu_i-\sum_j a_{ij}\sigma_j\Big)h_i,
\end{align*}
by \er{eq:parameters}. 
Applying the fundamental weights gives the mass reflection of \cite{LYZ}, after the normalization \er{eq:LYZ-normalization}: 
\begin{equation}\label{eq:mass-reflection}
 (R_i\sigma)_j=\sigma_j\quad(j\ne i),\qquad
 (R_i\sigma)_i=\sigma_i+\mu_i-\sum_j a_{ij}\sigma_j.
\end{equation}
Thus $\Gamma_A(\mu)$ is exactly the orbit of zero under the affine reflections \er{eq:mass-reflection}. 

Because $p_\mu$ lies in the open positive chamber, its stabilizer in $W$ is trivial. Hence $\Gamma_A(\mu)$ has $|W|$ distinct vectors. Each has nonnegative entries: $\omega_i-w^{-1}\omega_i$ is a nonnegative integral combination of simple roots, and
\[
 \omega_i(p_\mu-wp_\mu)=(\omega_i-w^{-1}\omega_i)(p_\mu)
\]
is accordingly a nonnegative integral combination of the positive $\mu_j$.

\subsection{The quadratic restriction}\label{sec:pohozaev}

Fix a $W$-invariant positive definite inner product on $\h_{\R}$, and extend it complex-bilinearly. With the convention \eqref{eq:cartan},
\begin{equation*}
 \ell_i=\tfrac12\norm{h_i}^2,\qquad
 \ip{h_i}{H}=\ell_i\alpha_i(H),\qquad
 \ip{h_i}{h_j}=\ell_i a_{ij}.
\end{equation*}
Thus $DA$ is the coroot Gram matrix of the inner product, where $D=\operatorname{diag}(\ell_1,\ldots,\ell_n)$.  

Choose the quadratic invariant $I_2(H)=\norm{H}^2$.
Applying \er{eq:invariant-constraints} with $H=p_\gamma$ and
$H'=p_\gamma-q_\sigma$, and using $p_\mu=p_\gamma+\rh$ from
\er{eq:parameters}, gives
\begin{equation}\label{eq:c2}
 I_2(p_\mu-q_\sigma)=I_2(p_\mu).
\end{equation}
Expanding \er{eq:c2}, we obtain
\[
 \norm{q_\sigma}^2=2\ip{p_\mu}{q_\sigma}.
\]
Since $q_\sigma=\sum_i\sigma_i h_i$, $p_\mu=\sum_i\mu_iE_i$, and $\ip{E_i}{h_j}=\ell_j\delta_{ij}$, this is
\begin{equation}\label{eq:PI-intro}
 \boxed{\ \sigma^TDA\sigma=2\mu^TD\sigma.\ }
\end{equation}
For type $A_n$, normalize $D=\Id$. Then \er{eq:PI-intro} becomes
\beq\lb{eq:An-quadratic}
 \sum_i\sigma_i^2-\sum_{i=1}^{n-1}\sigma_i\sigma_{i+1}
 =\sum_i\mu_i\sigma_i.
\eeq
Under \eqref{eq:LYZ-normalization}, \er{eq:An-quadratic} is precisely the Pohozaev identity \cite[(1.19)]{LYZ}. The symmetrizer $D$ gives the corresponding quadratic form in the other types. Unlike the higher residue constraints used here, the Pohozaev identity also holds under the variable-coefficient hypotheses of \cite{LYZ}.

\section{Examples and explicit constraints}\label{sec:examples}

The bounded boundary oscillation condition \er{eq:boundary} is
essential to both our result and that of \cite{LYZ}. We first
write out Chen's example \cite{Chen}, which violates this
condition and violates the Pohozaev constraint, while still
satisfying the orbit-plus-integer restriction.

The hypotheses \er{eq:isolated}, \er{eq:boundary}, and
\er{eq:mass-bound} still allow several concentrating centers to
collapse to zero. Our choice of $r_k$ in Lemma~\ref{lem:radii}
encloses the whole local mass. The example in Section~\ref{sec:rank-one} illustrates this possibility.

Finally, Section~\ref{sec:A2} shows the additional restriction
supplied by the cubic invariant.

\subsection{Chen's example with unbounded boundary oscillation}\label{sec:chen}
Chen's construction \cite[Theorems~1--3]{Chen} gives two collapsing bubbles with no prescribed source. We write it in our normalization and check directly that \eqref{eq:boundary} fails.

Following \cite[p.~300]{Chen}, choose an entire function $g$ with
\[
 g'(\zeta)=\frac{e^{\pi i(1-\zeta)}+1-2\zeta}{\zeta(\zeta-1)},
 \qquad f(\zeta)=e^{g(\zeta)}\zeta(\zeta-1).
\]
The apparent singularities of $g'$ at $0$ and $1$ are removable. The function $f$ has exactly two simple zeros, at $0$ and $1$, while
\[
 f'(\zeta)=e^{g(\zeta)+\pi i(1-\zeta)}\ne0.
\]
For $\lambda>0$, put
\[
 D_\lambda(\zeta)=\frac{\lambda^2|f'(\zeta)|^2}
                         {(1+\lambda^2|f(\zeta)|^2)^2}.
\]
On each fixed disk containing $0$ and $1$, $D_\lambda dA\rightharpoonup\pi(\delta_0+\delta_1)$: a change of variables near each simple zero gives mass $\pi$, and $D_\lambda\to0$ uniformly away from the zeros.

Choose $\lambda_k\to\infty$ sufficiently fast that
\begin{equation}\label{eq:chen-scale}\begin{gathered}
 \left|\int_{B_k}D_{\lambda_k}\dd-2\pi\right|\le k^{-1},\\
 \sup_{\overline B_k\setminus(B_{1/4}(0)\cup B_{1/4}(1))}
 D_{\lambda_k}\le k^{-4},\\
\lambda_k\min_{|\zeta|=k}|f(\zeta)|\to\infty.
\end{gathered}\end{equation}
Rescale by
\begin{equation*}
 e^{u^k(z)}=k^2D_{\lambda_k}(kz),\qquad |z|\le1.
\end{equation*}
These are smooth solutions of $\Delta u^k+8e^{u^k}=0$, since $f'$ never vanishes. Their disk masses tend to $2\pi$, their densities tend to zero on compact subsets of $\overline B_1\setminus\{0\}$, and the two centers $0$ and $1/k$ collapse to zero. Consequently $\sigma=2$, although $\gamma=0$ and the Weyl set is $\{0,1\}$. The mass belongs to $\{0,1\}+\mathbb Z$, so it does not contradict the orbit-plus-integer structure in \cite{LYZ}. It violates the Pohozaev identity $\sigma^2=\sigma$, the rank-one case of \er{eq:An-quadratic}. The boundary-oscillation hypothesis required for that identity fails, as we now check.

The boundary growth condition in \er{eq:chen-scale} gives,
uniformly on $|z|=1$,
\[
 u^k(z)=2\log k-2\log\lambda_k
       +2\log|f'(kz)|-4\log|f(kz)|+o(1).
\]
Since
\[
 \frac{f'(\zeta)}{f(\zeta)^2}
 =\frac{e^{-g(\zeta)+\pi i(1-\zeta)}}{\zeta^2(\zeta-1)^2},
\]
putting $G(\zeta)=-g(\zeta)+\pi i(1-\zeta)$ yields
\beq\lb{eq:chen-boundary}
 u^k(z)=2\log k-2\log\lambda_k
       +2\operatorname{Re}G(kz)-4\log|kz(kz-1)|+o(1).
\eeq
The term $-4\log|kz(kz-1)|$ in \er{eq:chen-boundary} has
oscillation $4\log((k+1)/(k-1))=O(k^{-1})$ on this circle.

On the other hand, $g'(0)=2-\pi i$ by the L'Hopital rule, so $G'(0)=-2$.
Thus, for $F_k(\theta)=2\operatorname{Re}G(ke^{i\theta})$,
the first Fourier coefficient is
\[
 \frac1{2\pi}\int_0^{2\pi}F_k(\theta)e^{-i\theta}\,d\theta=-2k.
\]
Let $a_k=(\max F_k+\min F_k)/2$. Subtracting this constant
does not change the Fourier coefficient, and hence
\[
 2k
 \le \frac1{2\pi}\int_0^{2\pi}|F_k(\theta)-a_k|\,d\theta
 \le \frac12\osc F_k.
\]
The remaining terms in \er{eq:chen-boundary} are constant
on $\p B_1$ or have vanishing oscillation. Therefore
\[
 \osc_{\p B_1}u^k\ge 4k-O(k^{-1})-o(1)\longrightarrow\infty.
\]

Thus \eqref{eq:isolated} and \eqref{eq:mass-bound} hold, but
\eqref{eq:boundary} fails. In particular, the harmonic-term
estimate in Lemma~\ref{lem:green} cannot be deduced from the
mass bound alone.

\subsection{Two selection centers}\label{sec:rank-one} 
For $A=(2)$, $U=u/2$, $h=\diag(1, -1)$, $p_\gamma=(\gamma/2)h$ and $p_\mu = (\mu/2)h$ with $\mu=1+\gamma>0$. 

The two Weyl-group elements give $p_\mu-p_\mu=0$ and $p_\mu-s_1p_\mu=2p_\mu=\mu h$, so 
$$
\sigma \in \{0, \mu\}.
$$

For example, the standard singular radial solution
\[
 e^{u(r)}=\frac{\mu^2\lambda^2 r^{2\mu-2}}
                    {(1+\lambda^2r^{2\mu})^2}
\]
solves $\Delta u+8e^u=4\pi\gamma\delta_0$ and has $\pi^{-1}\int_\C e^u=\mu$, directly checking the normalization.

The bounded-oscillation hypothesis allows more than one concentrating center. For example, take $\gamma=1$, set $a_k=k^{-1}$ and $\lambda_k=k^3$, and define
\begin{equation}\label{eq:two-center}
 e^{u^k(z)}=
 \frac{4\lambda_k^2|z|^2}
      {(1+\lambda_k^2|z^2-a_k^2|^2)^2}.
\end{equation}
The density in \er{eq:two-center} is the pullback density $|F_k'|^2/(1+|F_k|^2)^2$ for $F_k(z)=\lambda_k(z^2-a_k^2)$. Since $F_k'$ has a simple zero only at zero, \eqref{eq:two-center} solves
\[
 \Delta u^k+8e^{u^k}=4\pi\delta_0.
\]
The map $F_k$ has degree two, so change of variables gives $\int_\C e^{u^k}\dd=2\pi$. The two centers are $a_k$ and $-a_k$. Their common bubble scale is $\varepsilon_k=(2\lambda_k a_k)^{-1}$, and
\[
 \varepsilon_k^2e^{u^k(\pm a_k+\varepsilon_k y)}
 \longrightarrow\frac1{(1+|y|^2)^2}
 \quad\text{locally uniformly on }\C.
\]
This follows from $\varepsilon_k/a_k\to0$, $F_k(\pm a_k+\varepsilon_k y)\to\pm y$, and $\varepsilon_kF_k'(\pm a_k+\varepsilon_k y)\to\pm1$. Each bubble has mass $1$, so the two bubbles exhaust the mass $2$. Away from zero, $e^{u^k}\to0$ locally uniformly. On $|z|=1$,
\[
 1-a_k^2\le|z^2-a_k^2|\le1+a_k^2
 \quad\Longrightarrow\quad \osc_{\p B_1}u^k=O(a_k^2).
\]
Thus Theorem~\ref{thm:main} applies and gives the total mass $\sigma=2=1+\gamma$.

\subsection{The cubic constraint in type \texorpdfstring{$A_2$}{A2}}\label{sec:A2}
We use this example to illustrate that the quadratic and cubic residues together determine the six Weyl masses. We intentionally use the identity \er{cp} in this case. 

With the usual notation $\mu_i = 1+\gamma_i>0$, we have 
$$
p_\mu = \mu_1 E_1 + \mu_2 E_2 = \tfrac{1}{3}\diag(2\mu_1+\mu_2, -\mu_1+\mu_2, -\mu_1-2\mu_2).
$$
The Weyl group is the symmetric group $S_3$ permuting the diagonal
entries. On a traceless diagonal matrix, $\omega_1$ is the first
entry and $\omega_2$ is minus the third entry. Thus
\er{eq:Gamma} gives 
\begin{equation}\label{eq:A2-set}
 \Gamma_{A_2}(\mu)=
 \left\{\begin{aligned}
 &(0,0),\ (\mu_1,0),\ (0,\mu_2),\\
 &(\mu_1,\mu_1+\mu_2),\ (\mu_1+\mu_2,\mu_2),\\
 &(\mu_1+\mu_2,\mu_1+\mu_2)
 \end{aligned}\right\}.
\end{equation}

In type $A_2$, the scalar Miura factorization gives another
construction of the $W$-invariants; see \cite{Nie}.
We use the following order of factors:
\begin{gather}
 \bigl(\p-\p U_2\bigr)
 \bigl(\p+\p U_2-\p U_1\bigr)
 \bigl(\p+\p U_1\bigr)
 =\p^3+W_2\p+W_3,\label{eq:A2-Miura}\\
 W_2=\p^2U_1+\p^2U_2-(\p U_1)^2-(\p U_2)^2
       +(\p U_1)(\p U_2),\label{eq:A2-W2}\\
 W_3=\p^3U_1-2(\p U_1)(\p^2U_1)
       +(\p U_1)(\p^2U_2)\label{eq:A2-W3}\\
 \qquad-(\p U_1)(\p U_2)^2
       +(\p U_1)^2(\p U_2).\notag
\end{gather}
Substituting $\p U_i=\gamma^i/z$
into \er{eq:A2-W2} and \er{eq:A2-W3} gives the residue polynomials
\begin{equation}\label{eq:A2-residues}
 \begin{aligned}
 c_2(\gamma^1,\gamma^2)
 &=-\gamma^1-\gamma^2-(\gamma^1)^2-(\gamma^2)^2
   +\gamma^1\gamma^2,\\
 c_3(\gamma^1,\gamma^2)
 &=\gamma^1(1+\gamma^1-\gamma^2)(2+\gamma^2).
 \end{aligned}
\end{equation}

The conservation identities \er{cp} require both polynomials in \er{eq:A2-residues}
to be unchanged when $\gamma^i$ is replaced by
$\gamma^i-\sigma_i$. The quadratic equation is the $A_2$ case of \er{eq:An-quadratic}:
\begin{equation}\label{eq:A2-quadratic}
 \begin{aligned}
 \sigma_1^2-\sigma_1\sigma_2+\sigma_2^2
 &=(2\gamma^1-\gamma^2+1)\sigma_1 + (2\gamma^2-\gamma^1+1)\sigma_2\\
 &=\mu_1\sigma_1+\mu_2\sigma_2.
 \end{aligned}
\end{equation}

We omit writing out the $c_3$ condition. More neatly, the indicial polynomial of \er{eq:A2-Miura} factors as
\begin{equation*}
 t(t-1)(t-2)+c_2t+c_3
 =(t+\gamma^1)(t-1-\gamma^1+\gamma^2)(t-2-\gamma^2).
\end{equation*}
Thus conserving $c_2,c_3$
says that replacing $\gamma^i$ by $\gamma^i-\sigma_i$ in the above indices of 
$$
-\gamma^1,\ \gamma^1-\gamma^2+1,\ \gamma^2+2
$$
gives a permutation, again producing \er{eq:A2-set}.

For example, take $\mu=(1,4)$ and $\sigma=(3,1)$. By \er{eq:A2-set}, $(1,0)\in\Gamma_{A_2}(1,4)$, so the decomposition
\[
 (3,1)=(1,0)+(2,1),\qquad (2,1)\in\mathbb Z^2,
\]
verifies the orbit-plus-integer restriction. Moreover, \er{eq:A2-quadratic} holds since $9-3+1=3+4$. But $(\gamma^1,\gamma^2)=(1,2)$, and \er{eq:A2-residues} gives
\[
 (c_2,c_3)(1,2)=(-6,0),\qquad
 (c_2,c_3)(1-3,2-1)=(-6,12).
\]
Thus this mass satisfies the orbit-plus-integer and Pohozaev restrictions in \cite{LYZ}, but is here excluded for constant coefficients by the cubic invariant.

\end{document}